\documentclass[12pt]{amsart}
\usepackage{amsmath,amsfonts,amsthm,amssymb,mathrsfs,amsxtra,amscd,latexsym}
\usepackage[hmargin=2cm,vmargin=3cm]{geometry}
\usepackage[usenames]{color}
\usepackage[all]{xy}
 \newtheorem{thm}{Theorem}[section]

 \newtheorem{lem}[thm]{Lemma}
 \newtheorem{prop}[thm]{Proposition}
 
 \newtheorem{dfn}[thm]{Definition}

 \newtheorem{rmk}[thm]{Remark}
 
 \newtheorem*{thmA}{Theorem A}
 \newtheorem*{thmB}{Theorem B}
 \theoremstyle{definition}
 \theoremstyle{remark}
 \numberwithin{equation}{section}

\newcommand{\sm}{\left(\begin{smallmatrix}}
\newcommand{\esm}{\end{smallmatrix}\right)}
\newcommand{\mat}{\left(\begin{matrix}}
\newcommand{\emat}{\end{matrix}\right)}

\def\CC{\mathbb{C}}
\def\HH{\mathbb{H}}

\def\QQ{\mathbb{Q}}

\def\ZZ{\mathbb{Z}}

\def\SL{\mathrm{SL}}

\begin{document}

\title[Schneider-Siegel theorem for harmonic Maass forms]{Schneider-Siegel theorem for a family of values of a harmonic weak Maass form at Hecke orbits}


\author{Dohoon Choi}
\author{Subong Lim}

\address{Department of Mathematics, Korea University, 145 Anam-ro, Seongbuk-gu, Seoul 02841, Republic of Korea}
\email{dohoonchoi@korea.ac.kr}

\address{Department of Mathematics Education, Sungkyunkwan University, Jongno-gu, Seoul 03063, Republic of Korea}
\email{subong@skku.edu}



\begin{abstract}
Let $j(z)$ be the modular $j$-invariant function. Let $\tau$ be an algebraic number in the complex upper half plane $\mathbb{H}$.  It was proved by  Schneider and Siegel that if $\tau$ is not a CM point, i.e., $[\mathbb{Q}(\tau):\mathbb{Q}]\neq2$, then $j(\tau)$ is transcendental.
Let $f$ be a harmonic weak Maass form of weight $0$ on $\Gamma_0(N)$.
In this paper, we consider an extension of the results of Schneider and Siegel to a family of values of  $f$   on Hecke orbits of $\tau$. 
%

For a positive integer $m$, let  $T_m$ denote the $m$-th Hecke operator. 
Suppose that the coefficients of the principal part of $f$ at the cusp $i \infty$ are algebraic, and that $f$ has its poles only at cusps equivalent to $i \infty$. 
We prove,  under a mild assumption on $f$,  that  for any fixed $\tau$, if $N$ is a prime such that
$ N\geq 23 \text{ and }  N \not \in \{23, 29, 31, 41, 47, 59, 71\},$
then $f(T_m.\tau)$ are transcendental for infinitely many positive integers $m$ prime to $N$.
\end{abstract}

\maketitle

\section{Introduction} \label{section1}
Let $j(z)$ be the modular $j$-invariant function on the complex upper half plane $\mathbb{H}$. Let $\tau$ be an algebraic number in $\mathbb{H}$. It was proved by Kronecker \cite{Krone} and Weber \cite{Web} that if $\tau$ is a CM point, i.e., $[\mathbb{Q}(\tau):\mathbb{Q}]=2$, then $j(\tau)$ is algebraic. Schneider \cite{Schn} and Siegel \cite{Si}  proved that if $\tau$ is not a CM point, then $j(\tau)$ is transcendental.
By combining these two results, we state the following.

\begin{thmA}[Kronecker,  Schneider, Siegel, Weber]
Assume that $\tau$ is an algebraic number in $\mathbb{H}$.
Then, $\tau$ is a CM point if and only if $j(\tau)$ is algebraic.
\end{thmA}

Let $m$ be a positive integer, and $T_m$ denote the $m$-th Hecke operator. The operators $T_m$ act on both of modular forms $f$ and divisors $D$ of a modular curve, and they are denoted by $f|T_m$ and $T_m.D$, respectively.
Then, $j(T_m.\tau)=(j|T_{m})(\tau)$, and $(j|T_{m})(z)$ is a polynomial of $j(z)$ with rational coefficients. Thus, $j(T_m.\tau)$ is algebraic for every $m$ if and only if $j(\tau)$ is algebraic. Therefore, Theorem A is equivalent to the following theorem.

\begin{thmB}[Kronecker, Schneider,  Siegel, Weber]
Assume that $\tau$ is an algebraic number in $\mathbb{H}$.
Then, $\tau$ is a CM point if and only if $j(T_m.\tau)$ is algebraic for every positive integer $m$ .
\end{thmB}

In this vein, we consider an extension of the results of Kronecker,  Schneider, Siegel, and Weber to a family of values of  a harmonic weak Maass form $f$ on Hecke orbits of $\tau$. Let $N$ be a positive integer, and $f$ be a harmonic weak Maass form of weight $0$ on $\Gamma_0(N)$. In contrast to the case for the $j$-invariant function, the value of $f$ at a CM point $\tau$ is not algebraic in general. Thus, first we obtain the period of $f(\tau)$ for a CM point $\tau$, which is expressed as the regularized Petersson inner product of a cusp form and a meromorphic modular form. Next, by using this result, we obtain an extension of the results of Kronecker,  Schneider, Siegel, and Weber to a family of values of $f$ on Hecke orbits of $\tau$.

Let $Y_0(N)$ be the modular curve of level $N$ defined by $\Gamma_0(N)\backslash\mathbb{H}$, and $X_0(N)$ denote the compactification of $Y_0(N)$ by adjoining the cusps. Let us note that $X_0(N)$ is a curve defined over $\mathbb{Q}$. We fix an algebraic closure $\overline{\mathbb{Q}}$ of $\mathbb{Q}$. Let $K$ be a subfield of $\overline{\mathbb{Q}}$ or the field  $\CC$ of complex numbers. Let $C$ be a curve defined over $K$.
For an extension $E$ of $K$, we denote by $\mathrm{Div}_C(E)$ the group of divisors of $C$ defined over $E$.
Let $f$ be a function on $C(\mathbb{C})\setminus S$ for a finite subset $S$ of $C(\mathbb{C})$.
If $D=\sum_{P \in C} n_P P\in \mathrm{Div}_C(\mathbb{C})$ and the support of $D$ does not contain any point in $S$, then we define
\[
f(D):=\sum n_P f(P).
\]
The $m$-th Hecke operator $T_m$ acts on $\mathrm{Div}_{X_0(N)}(\mathbb{C})$, and it is denoted by $T_m.D$ for $D\in \mathrm{Div}_{X_0(N)}(\mathbb{C})$.

Let $k$ be a non-negative even integer.
Let $S_{k}(\Gamma_0(N))$ denote the space of cusp forms of weight $k$ on $\Gamma_0(N)$.
We denote by $H_{k}(\Gamma_0(N))$ the space of harmonic weak Maass forms of weight $k$  on $\Gamma_0(N)$.
For the differential operator $\xi_{-k}$ defined by  $\xi_{-k}(f)(z) := 2iy^{-k}\overline{\frac{\partial }{\partial\bar{z}}f(z)}$,  the assignment $f(z)\mapsto \xi_{-k}(f)(z)$ gives an anti-linear mapping
\begin{equation} \label{xi}
\xi_{-k}: H_{-k}(\Gamma_0(N)) \to M^!_{k+2}(\Gamma_0(N)),
\end{equation}
where $M^!_{k}(\Gamma_0(N))$ denotes the space of weakly holomorphic modular forms of weight $k$ on $\Gamma_0(N)$.
Here, $y$ denotes the imaginary part of $z\in \HH$.
Let $H^*_{-k}(\Gamma_0(N))$ be the inverse image of the space $S_{k+2}(\Gamma_0(N))$  of cusp forms under the mapping $\xi_{-k}$.

\begin{dfn}
Let $f$ be a harmonic weak Maass form of weight $0$ on $\Gamma_0(N)$. We say that $f$ is {\it arithmetic} if $f$ satisfies the following conditions:
\begin{enumerate}
\item the principal part of $f$ at the cusp $i \infty$ belongs to $\overline{\mathbb{Q}}[q^{-1}]$, and its constant term is zero,
\item the principal part of $f$ at each cusp not equivalent to $i \infty$ is constant.
\end{enumerate}
Here, $q := e^{2\pi iz}$ for a complex number $z\in \HH$.
\end{dfn}

This definition is similar with that of being good  in \cite{BOR}; however, the conditions in this definition are weaker than those in the definition of being good. Furthermore, a harmonic weak Maass form $f$ is called an {\it arithmetic Hecke eigenform} if $f$ is arithmetic and $ \xi_{0}(f)$ is a Hecke eigenform.

For $\tau \in \mathbb{H} \cup \{ i \infty \} \cup \mathbb{Q}$, let $Q_{\tau}$ be the image of $\tau$ under the canonical map from $ \mathbb{H} \cup \{ i \infty \} \cup \mathbb{Q}$ to $X_0(N)$ and
\begin{equation*} \label{D_tau}
D_{\tau}:=Q_{i \infty}-Q_{\tau}\in \mathrm{Div}_{X_0(N)}(\mathbb{C}).
\end{equation*}
If $\tau$ is a CM point, then $D_{\tau}$
is defined over $\overline{\mathbb{Q}}$.
Thus, there exists a differential $\psi_{D_\tau}^{alg}$ of the third kind associated to $D_{\tau}$ such that $\psi_{D_\tau}^{alg}$ is defined over $\overline{\mathbb{Q}}$.
Note that $\psi_{D_\tau}^{alg}$ can be written as $\psi_{D_\tau}^{alg} = 2\pi if_{\psi_{D_\tau}^{alg}}(z)dz$ for some meromorphic modular form $f_{\psi_{D_\tau}^{alg}}$ of weight $2$ on $\Gamma_0(N)$ (see Section \ref{Differentials} for details). Let $\left( \xi_{0}(f), f_{\psi_{D_\tau}^{alg}} \right)_{reg}$ be the regularized Petersson inner product of $\xi_{0}(f)$ and $f_{\psi_{D_\tau}^{alg}}$ (see Section \ref{Regularized} for the definition of the regularized Petersson inner product). The following theorem shows that, for each positive integer $m$ prime to $N$, the period of $f(T_m.Q_{\tau})$ can be expressed as the multiplication of $\left( \xi_{0}(f), f_{\psi_{D_\tau}^{alg}} \right)_{reg}$ and the eigenvalue of $\xi_{0}(f)$ for $T_m$.

\begin{thm} \label{main0}
Let $N$ be a prime, and  $f$ be a harmonic weak Maass form of weight $0$ on $\Gamma_0(N)$.
Assume that $f$ is an arithmetic Hecke eigenform, and that $\tau$ is a CM point.
Let $\psi_{D_\tau}^{alg} := 2\pi i f_{\psi_{D_\tau}^{alg}}(z)dz$ be a differential of the third kind associated to $D_\tau$ defined over $\overline{\mathbb{Q}}$. Then,
\[
f(T_m.Q_{\tau})-m^{-1}\lambda_m \left(\xi_{0}(f), f_{\psi_{D_\tau}^{alg}} \right)_{reg}
\]
is algebraic for every positive integer $m$ prime to $N$, where $\lambda_m$ is the eigenvalue of $\xi_0(f)$ for $T_m$.
\end{thm}

Theorem \ref{main0} is applied to study the transcendence of $f(T_m.Q_{\tau})$ for an algebraic number $\tau$ in $\mathbb{H}$.
Then, we have the following theorem concerning to an extension of the above results of Kronecker,  Schneider, Siegel, and Weber to a family of values of $f$ on Hecke orbits of $\tau$.

\begin{thm}\label{main1}
Let $N$ and $f$ be given as in Theorem \ref{main0}.
Assume  that $(g, \xi_{0}(f)) \neq 0$ for each Hecke eigenform $g \in S_2(\Gamma)$. Let $\tau$ be an algebraic number in $\mathbb{H}$.
Then, $f(T_m.Q_\tau)$ is algebraic for every positive integer $m$ prime to $N$ if and only if $\tau$ is a CM point and $n D_\tau$ is rational on $X_0(N)$ for some positive integer $n$.
\end{thm}

\begin{rmk}
Assume that  $f$ and $\tau$ are given as in Theorem \ref{main1}.
In fact, if there is a positive integer $m$ prime to $N$ such that $f(T_m . Q_\tau)$ is transcendental, then
there are infinitely many such positive integers $m$ prime to $N$ (see the proof of Theorem \ref{main1} in Section \ref{Proof}).
\end{rmk}

Let $J_{\Gamma_0(N)}$ be the Jacobian variety of the $X_0(N)$ defined over $\mathbb{Q}$. Then, an Albanese embedding $i_{Q_{i\infty}}:X_0(N) \rightarrow J_{\Gamma_0(N)}$ can be defined by sending $Q$ to  $Q_{i \infty}-Q$. Let us note that $m(Q_{i \infty}-Q)$  is rational on $X_{0}(N)$ for some positive integer $m$ if and only if $i_{Q_{i\infty}}(Q)$ is a torsion point in $J_{\Gamma_0(N)}$.
Let
\[
T_{Q_{i \infty}}(X_{0}(N)):=\left\{ Q \in X_{0}(N)(\overline{\mathbb{Q}}) \; \big| \; i_{Q_{i\infty}}(Q) \text{ is a torsion point in } J_{\Gamma_0(N)}\right\}.
\]
If the genus of $X_{0}(N)$ is larger than or equal to $2$, then, by the Mumford-Manin conjecture (proved by Raynaud), $T_{Q_{i \infty}}(X_{0}(N))$ is a finite set.

Let $X^+_0(N)$ denote the quotient of $X_0(N)$ by the Atkin-Lehner involution $w_N$.
For primes $N$, Coleman, Kaskel, and Ribet \cite{CKR} conjectured the following statement:
for all prime numbers $N \geq  23$,
\[
T_{Q_{i \infty}}(X_{0}(N)) =
\left \{
\begin{array}{ll}
  \{0,i \infty\} & \text{ if } g^+ >0, \\
  \{0,i \infty\}\cup \{\text{hyperelliptic branch points}\} & \text{ if } g^+=0,
\end{array} \right.
\]
where $g^+$ denotes the genus of $X_0(N)^+$.
Baker \cite{Bak} proved this conjecture.
Furthermore, for $N\geq23$, $g^+$ is zero if and only if $N \in \{23, 29, 31, 41, 47, 59, 71\}$.
Thanks to these results on torsion points on the Jacobian of a modular curve,
we obtain the following theorem from Theorem \ref{main1}.

\begin{thm} \label{main2}
Under the assumption as in Theorem \ref{main1}, assume that $$N\geq 23 \text{ and }  N \not \in \{23, 29, 31, 41, 47, 59, 71\}.$$
Then, $f(T_m.Q_\tau)$ are transcendental for infinitely many positive integers $m$ prime to $N$.
\end{thm}

The remainder of this paper is organized as follows.
In Section \ref{Preliminaries}, we introduce some preliminaries for harmonic weak Maass forms, residues of meromorphic differentials on a modular curve, and differentials of the third kind on a complex curve.
In Section \ref{Regularized}, we review  the definition of a regularized Petersson inner product and prove that the regularized Petersson inner product of a meromorphic modular form, associated with a canonical differential of the third kind of some divisor on $X_0(N)$, with every cusp form of weight $2$ on $\Gamma_0(N)$ is zero.
In Section \ref{Proof}, we prove Theorem \ref{main0} and  \ref{main1}.

\section{Preliminaries} \label{Preliminaries}
In this section, we recall definitions and basic facts about harmonic weak Maass forms, residues of meromorphic differentials on a modular curve, and properties for differentials of the third kind on a complex curve.

\subsection{Harmonic weak Maass forms}
For details of harmonic weak Maass forms, we refer to \cite{BF} and \cite{O2}.
Let $k$ be an even integer.
We  recall the weight $k$ slash operator
\[(f|_{k}\gamma)(z) := (cz+d)^{-k}f(\gamma z)\]
for any function $f$ on $\HH$ and $\gamma = \sm a&b\\c&d\esm\in \SL_2(\ZZ)$. Let $\Delta_k$ denote the weight $k$ hyperbolic Laplacian defined by
\[
\Delta_k := -y^2 \left (\frac{\partial^2}{\partial x^2} + \frac{\partial^2}{\partial y^2}\right ) + iky \left (\frac{\partial}{\partial x} + i\frac{\partial}{\partial y} \right ),
\]
where $x$ (resp. $y$) denotes the real (resp. imaginary) part of $z$.
Now, we give the definition of a harmonic weak Maass form.

\begin{dfn}
Let $N$ be a positive integer.
A smooth function $f$ on $\HH$ is a harmonic weak Maass form of weight $k$ on $\Gamma_0(N)$ if
it satisfies the following conditions:
\begin{enumerate}
\item[(1)] $f|_{k}\gamma = f$ for all $\gamma\in\Gamma_0(N)$,
\item[(2)] $\Delta_k f =0$,
\item[(3)] a linear exponential growth condition in terms of $y$  at every cusp of $\Gamma_0(N)$.
\end{enumerate}
We denote by $H_{k}(\Gamma_0(N))$  the space of harmonic weak Maass forms of weight $k$  on $\Gamma_0(N)$.
\end{dfn}

Assume that $t$ is a cusp of $\Gamma_0(N)$. Let $\sigma_{t}\in \mathrm{SL}_2(\mathbb{Z})$ be a matrix such that $\sigma_{t} (i \infty)= t$, and $\Gamma_0(N)_{t}$ denote the stabilizer of cusp $t$ in $\Gamma_0(N)$. We define a positive integer $\alpha_t$ by
\[
\sigma_{t}^{-1} \Gamma_0(N)_{t} \sigma_{t}=\left \{\pm \left(
                                                     \begin{smallmatrix}
                                                       1 & \ell \alpha_t \\
                                                       0 & 1 \\
                                                     \end{smallmatrix}
                                                   \right)
     \; : \; \ell \in \mathbb{Z} \right\}.
\]
Recall that $f|_k \sigma_t$  has the Fourier expansion of the form $f|_k \sigma_t =f^+_t+f^-_t$, where
\begin{align} \label{decomposition}
f^+_t(z) &= \sum_{n\gg-\infty} a_f^t(n) e^{2\pi inz/\alpha_t},\\
\nonumber f^-_t(z) &= b^t_{f}(0)y^{1-k} + \sum_{n\ll \infty\atop n\neq 0} b^t_f(n)  \Gamma(4\pi n y/\alpha_t,-k+1 )e^{2\pi inz/\alpha_t},
\end{align}
where $\Gamma(x,s)$ denotes the incomplete gamma function defined as an analytic continuation of the function $\int_{x}^\infty t^{s-1}e^{-t}dt$.
This Fourier series is called {\it the Fourier expansion of $f$ at a cusp $t$}.
The function $\sum_{n\leq 0} a^t_f(n)e^{2\pi inz/\alpha_t}$ is called {\it the principal part} of $f$ at the cusp $t$.

For a positive integer $n$, let $T_n$ denote the $n$-th Hecke operator.
Then, the Hecke operator $T_n$ commutes with the differential operator $\xi_{-k}$ in the following way
\begin{equation}\label{Heckedifferentialcommutative}
\xi_{-k}(f|_{-k}T_n) = n^{-k-1}(\xi_{-k}(f)|_{k+2}T_n)
\end{equation}
for a harmonic weak Maass form $f$ of weight $-k$.

\subsection{Residues of a meromorphic differential on $X_0(N)$}
Let $\psi$ be a meromorphic differential on $X_0(N)$. Then, there exists a unique meromorphic modular form $g$ of weight $2$ on $\Gamma_0(N)$ such that $\psi=g(z)dz$. Assume that $t$ is a cusp of $\Gamma_0(N)$. Assume that, for each cusp $t$, $g$ has the Fourier expansion of the form
\[
(g|_{2} {\sigma_t})(z)=\sum a_g^t(n)q^{n/\alpha_t},
\]
where $q := e^{2\pi iz}$ for $z\in\HH$.

For $\tau \in \mathbb{H} \cup \{ i \infty \} \cup \mathbb{Q}$, let $Q_{\tau}$ be the image of $\tau$ under the canonical map from $ \mathbb{H} \cup \{ i \infty \} \cup \mathbb{Q}$ to $X_0(N)$.
Let $\mathrm{Res}_{Q_{\tau}}gdz$ denote the residue of the differential $g(z)dz$ at $Q_{\tau}$ on $X_0(N)$, and $\mathrm{Res}_{\tau}g$ be the residue of $g$ at $\tau$ on $\mathbb{H}$.
We describe $\mathrm{Res}_{Q_z}gdz$ in terms of $\mathrm{Res}_{\tau}g$ as follows.
Let $\mathcal{C}_N$ be the set of inequivalent cusps of $\Gamma_0(N)$. For $\tau \in \mathbb{H}$, let $e_{\tau}$ be the order of the isotropy subgroup of $\mathrm{SL}_2(\mathbb{Z})$ at $\tau$.
Then, we have
\[
\mathrm{Res}_{Q_{\tau}}gdz=
\begin{cases}
  \frac{1}{e_{\tau}}\mathrm{Res}_{\tau}g & \text{ if } \tau \in \mathbb{H}, \\
  \frac{1}{2\pi i}\alpha_{\tau} a^{\tau}_{g}(0) & \text{ if } \tau \in \mathcal{C}_N.
\end{cases}
\]

\subsection{Differentials of the third kind} \label{Differentials}

In this subsection, we review properties for differentials of the third kind on a complex curve. For details, we refer to \cite{BO}, \cite{Gri}, and \cite{Sch}.
A differential of the third kind on $X_0(N)$  means a meromorphic differential on $X_0(N)$ such that its poles are simple and its residues are integers.
Let $\phi$ be a differential of the third kind on $X_0(N)$ such that $\phi$ has a pole at $P_j$ with residue $m_j$ and is holomorphic elsewhere.

Then, we define a linear map $res$ for the space of differentials of the third kind on $X_0(N)$ to $\mathrm{Div}_{X_0(N)}(\CC)$ by
\[
res(\phi)=\sum_{j}m_j P_j.
\]
The image of $\phi$ under the map $res$ is called the residue divisor of $\phi$.
By the residue theorem, the residue divisor $res(\phi)$ has degree zero.

Conversely, if $D$ is a divisor on $X_0(N)$ whose degree is zero, then there is a differential $\psi_D$ of the third kind with $res(\psi_D) = D$
by the Riemann-Roch theorem and the Serre duality.
The differential $\psi_D$ is unique up to addition of a cusp form of weight $2$ on $\Gamma_0(N)$.

Let $D$ be a divisor of $X_0(N)$ with degree zero.
Then, there is a unique differential $\Phi_D$ of the third kind such that $res(\Phi_D)=D$ and $\Phi_D=\partial_z h$, where $h$ is a real harmonic function on $X_0(N)$ with some singularities. Here, the differential $\Phi_D$ of the third kind is called {\it the canonical differential of the third kind associated to $D$}
(for example, see Section 1 in \cite{Sch} for more details). Scholl proved the following theorem by using  Waldschmidt's result on the transcendence of periods of differentials of the third kind.

\begin{thm}\cite[Theorem 1]{Sch} \label{algebraic_differential}
With the above notation, assume that $D$ is defined over a number field $F$. Then, $\Phi_D$ is defined over $\overline{\mathbb{Q}}$ if and only if some non-zero multiple of $D$ is a principal divisor.
\end{thm}

For a differential $\psi$ of the third kind, we may write $\psi = 2\pi if(z)dz$, where $f$ is a meromorphic modular form of weight $2$ on $\Gamma_0(N)$.
All poles of $f$ are simple poles and lie on $Y_0(N)$, and their residues are integers.
The residue of $\psi$ at the cusp $t$ is a constant term of the Fourier expansion of $f$ at the cusp $t$.
By the $q$-expansion principle, $\psi$ is defined over a number field $F$ if and only if all Fourier coefficients of $f$ at the cusp $i\infty$ are contained in $F$.
Therefore, the following theorem \cite[Theorem 3.3]{BO} was followed from Theorem \ref{algebraic_differential}.

\begin{thm}\cite[Theorem 3.3]{BO}
Let $F$ be a number field. Let $D$ be a divisor of degree $0$ on $X_0(N)$ defined over $F$.
Let $\Phi_D$ be the
canonical differential of the third kind associated to $D$ and write $\Phi_D = 2 \pi if(z) dz$. If some
non-zero multiple of $D$ is a principal divisor, then all the coefficients $a(n)$ of $f$ at the cusp
$i\infty$ are contained in $F$. Otherwise, there exists an integer $n$ such that $a(n)$ is transcendental.
\end{thm}

\section{Regularized Petersson inner product} \label{Regularized}
Petersson introduced an inner product on the space of cusp forms, which is called the Petersson inner product.
Borcherds \cite{Bor} used a regularized integral to extend the Petersson inner product to the case that one of two modular forms is a weakly holomorphic modular form.
In this section, we recall the definition of a regularized Petersson inner product of a cusp form and a meromorphic modular form with the same weight by following \cite{Bor} and \cite{C1}.
Furthermore, we prove that if $g$ is a meromorphic modular form on $\Gamma_0(N)$ such that $2\pi i g(z)dz$ is the canonical differential of the third kind associated to some divisor, then the regularized Petersson inner product of $g$ with every cusp form of weight $2$ on $\Gamma_0(N)$ is zero.

Let $g$ be a meromorphic modular form of weight $k$ on $\Gamma_0(N)$.
Let $\mathrm{Sing}(g)$ be the set of singular points of $g$ on $\mathcal{F}_N$, where $\mathcal{F}_N$ denotes the fundamental domain for the action of $\Gamma_0(N)$ on $\HH$.
For a positive real number $\varepsilon$,  an $\varepsilon$-disk $B_{\tau}(\varepsilon)$ at $\tau$ is defined by
\[
B_{\tau}(\varepsilon):=\left\{
\begin{array}{ll}
  \{ z \in \mathbb{H} \; : \; |z-\tau| < \varepsilon \} & \text{if } \tau \in  \mathbb{H},  \\
  \{ z \in \mathcal{F}_N \; : \; \mathrm{Im}(\sigma_{\tau}z) > 1/\varepsilon \} & \text{if } \tau \in  \{ i \infty\} \cup \mathbb{Q}.
\end{array}
\right.
\]
Let $\mathcal{F}_N(g, \varepsilon)$ be a punctured fundamental domain for $\Gamma_0(N)$ defined by
\[
\mathcal{F}_N(g, \varepsilon):=\mathcal{F}_N-\cup_{\tau \in \mathrm{Sing}(g) \cup \mathcal{C}_N} B_{\tau}(\varepsilon).
\]
Let $f$ be a cusp form of weight $k$ on $\Gamma_0(N)$.
The regularized Petersson inner product $(f,g)_{reg}$ of $f$ and $g$ is defined by
\[
(f,g)_{reg}:=\lim_{\varepsilon \rightarrow 0} \int_{\mathcal{F}_N(g,\varepsilon)}f(z)\overline{g(z)}\frac{dxdy}{y^{k-2}}.
\]

By using the Stokes' theorem, we obtain the following theorem (for example, see \cite[Proposition 3.5]{BF} or \cite[Lemma 3.1]{C1}).

\begin{lem} \label{regularized_formula}
Suppose that $f$ is a harmonic weak Maass form in $H^*_{k}(\Gamma_0(N))$ with singularities only at cusps equivalent to $i\infty$, and that $g$  is a meromorphic modular form. 
Then
\[
(\xi_0(f),g)_{reg}=\sum_{m+n=0} \sum_{t \in \mathcal{C}_N} a_f^t(m) a_g^t(n)+\sum_{ \tau \in \mathcal{F}_N} \frac{2\pi i}{e_{\tau}}\mathrm{Res}_{\tau}(g)f(\tau),
\]
where $a_g^t(n)$ and $a_f^t(n)$ are $n$-th Fourier coefficients of $g$ and $f$ at the cusp $t$, respectively.
\end{lem}

By using this lemma, we prove the following theorem, which states that the regularized Petersson inner product of a meromorphic modular form associated with a canonical differential of the third kind of some divisor on $X_0(N)$ with every cusp form of weight $2$ on $\Gamma_0(N)$ is zero.

\begin{prop}\label{regpetersson}
Let $g$ be a meromorphic modular form of weight $2$ on $\Gamma_0(N)$ associated with a canonical differential of the third kind. Then, for every cusp form $f$ of weight $2$ on $\Gamma_0(N)$,
\[
(f,g)_{reg}=0
\]
\end{prop}

\begin{proof}
Let $G$ be a harmonic function on $\mathbb{H}$ with $\log$-type singularities such that $\partial_z G=g$.
Then, we have
\[
d(f(z) \overline{G(z)}dz)=f(z)  \overline{\partial_z G(z)} d\overline{z}dz =f(z) \overline{g(z)} (2i)dxdy.
\]

To apply the Stokes' theorem, we give the description of the boundary of $\mathcal{F}_N(g,\epsilon)$.
For a subset $D$ of $\mathbb{C}$, let $\partial D$ denote the boundary of $D$.
For a positive real number $\varepsilon$,  we define
\[
\gamma_{\tau}(\varepsilon):=\left\{
\begin{array}{ll}
  \{ z \in \mathbb{H} \; : \; |z-\tau| = \varepsilon \} & \text{if } \tau \in  \mathbb{H},  \\
  \{ z \in \mathcal{F}_N \; : \; \mathrm{Im}(\sigma_{\tau}z) = 1/\varepsilon \} & \text{if } \tau \in  \{ i \infty\} \cup \mathbb{Q}.
\end{array}
\right.
\]
Assume that $\varepsilon$ is sufficiently small. If we let $\partial^* \mathcal{F}_N(g, \varepsilon)$ be the closure of the set $\partial \mathcal{F}_N(g, \varepsilon)-\partial \mathcal{F}_N$ in $\mathbb{C}$, then
\begin{equation}\label{boundary}
\partial^* \mathcal{F}_N(g, \varepsilon)=\cup_{\tau \in \mathrm{Sing}(g) \cup \mathcal{C}_N} \gamma_{\tau}(\varepsilon).
\end{equation}
From (\ref{boundary}), the Stokes' theorem implies
\[
\int_{\mathcal{F}_N(g,\varepsilon)}f(z)\overline{g(z)}dxdy= \int_{\partial^* \mathcal{F}_N(g, \varepsilon)} \frac{1}{2i} f(z) \overline{G(z)}dz=\sum_{\tau \in \mathrm{Sing}(g) \cup \mathcal{C}_N} \int_{\gamma_{\tau}(\varepsilon)} \frac{1}{2 i} f(z) \overline{G(z)}dz.
\]

For each $\gamma \in \SL_2(\mathbb{Z})$, the absolute value $|(f|_2 {\gamma})(z)|$ exponentially decays as $\mathrm{Im}(z) \rightarrow \infty$, since $f$ is a cusp form. Thus, if $\tau \in \mathcal{C}_N$, then $\lim_{\varepsilon \rightarrow 0}\int_{\gamma_{\tau}(\varepsilon)} \frac{1}{ 2i} f(z) \overline{G(z)} d{z}=0$.

To complete the proof, we assume that $\tau \in \mathrm{Sing}(g)$. Then,
\begin{align*}
&\left | \int_{\gamma_{\tau}(\varepsilon)} \frac{1}{ 2i} f(z) \overline{G(z)} d{z} \right| \\
&\leq \int_{\gamma_{\tau}(\varepsilon)} \frac{1}{2} |f(z)| |\overline{G(z)}| |d{z}| \\
&\leq \max\{ |G(z)| \; : \; z \in \gamma_{\tau}(\varepsilon) \} M_1  \int_{\gamma_{\tau}(\varepsilon)}|d{z}| \;\; (\text{some constant } M_1) \\
&\leq  \max\{ |G(z)| \; : \; z \in \gamma_{\tau}(\varepsilon) \} M_1 (2\pi \varepsilon).
\end{align*}
The function $G$ can be expressed around $\tau$ as
\[
G(z)=-\log_e (|z-\tau|(z-\tau))+G_0(z),
\]
where $G_0(z)$ is a smooth function around $\tau$.
In the definition of $\log_e$, we use the principal branch. If $\varepsilon$ is sufficiently small, then, for any $z \in \gamma_{\tau}(\varepsilon)$,  we have
\begin{align*}
|G(z)|&\leq |\log_{e}(|z-\tau|(z-\tau))|+|G_0(z)| \\
&\leq |\log_{e}|z-\tau||+  |\log_{e}(z-\tau)| +      M_2 \;\; (\text{some fixed constant } M_2)\\
&\leq 2|\log_{e} \varepsilon|+ \pi +           M_2.
\end{align*}
Thus, for sufficiently small $\varepsilon$, we obtain
\[
\left | \int_{\gamma_{\tau}(\varepsilon)} \frac{1}{2i} |f(z)| \overline{G(z)}d{z} \right| \leq (2|\log_{e} \varepsilon|+ \pi + M_2)  M_1 (2\pi \varepsilon).
\]
This implies that, for $\tau \in \mathrm{Sing}(g)$,
\[
\lim_{\varepsilon \rightarrow 0}  \int_{\gamma_{\tau}(\varepsilon)} \frac{1}{2 i}f(z) \overline{G(z)}d{z} =0.
\]
Thus, we complete the proof.
\end{proof}

\section{Proofs} \label{Proof}
In this section, we prove Theorem \ref{main0} and \ref{main1}.
In this section, we always assume that $N$ is a prime.
Let $D$ be a divisor of $X_0(N)$.
Let $\psi_D$ be a differential of the third kind associated to $D$.
Let $A_D$ be the space of meromorphic differentials, with only simple poles, such that their poles are only at the support of $D$.
If $D$ is defined over $\overline{\mathbb{Q}}$, then
there is a basis of $A_D$ consisting of meromorphic differentials defined over $\overline{\mathbb{Q}}$. Let us note that $D$ is defined over $\overline{\mathbb{Q}}$ if and only if there exists $\psi_D$ defined over $\overline{\mathbb{Q}}$.
Thus, for a divisor $D$ defined over $\overline{\mathbb{Q}}$, let $\psi_D^{alg}$ be a differential of the third kind associated to $D$ defined over $\overline{\mathbb{Q}}$.
Let $\Phi_D$ be the canonical differential of the third kind associated to $D$ and
\[
F_D:=\Phi_D-\psi_D^{alg}.
\]
Note that $F_D$ (resp. $\Phi_D$ and $\psi_D^{alg}$) can be written as $2\pi if_{F_D}(z)dz$ (resp. $2\pi if_{\Phi_D}(z)dz$ and $2\pi if_{\psi_D^{alg}}(z)dz$) for some meromorphic modular forms $f_{F_D}, f_{\Phi_D}$, and $f_{\psi_D^{alg}}$ of weight $2$ on $\Gamma_0(N)$.

Let us consider  special divisors $D_\tau$, where 
\[
D_\tau = Q_{i\infty} - Q_{\tau} \in \mathrm{Div}_{X_0(N)}(\mathbb{C})
\] 
for $\tau\in \HH$.
Let us note that $Q_{i\infty}$  is defined over $\bar{\QQ}$.
A modular curve $Y_0(N)$ is defined by an equation $\Phi_N(X,Y) = 0$ such that $\Phi_N(X,Y) \in \mathbb{Q}[X,Y]$ and $\Phi_{N}(j(Nz), j(z)) = 0$ for all $z\in \HH$. 
Thus, by Theorem A, $\tau\in\HH$ is a CM point if and only if $Q_\tau$ is defined over $\bar{\QQ}$.
Thus, there exists $\psi_{D_\tau}^{alg}$  for $D_\tau$ if and only if $\tau$ is a CM point. 

With these notations, we prove the following lemma. 

\begin{lem}\label{algebraic}
With the above notation, assume that $f$ is an arithmetic harmonic weak Maass form in $H^*_{k}(\Gamma_0(N))$, and that $D:= Q_{i\infty} - Q_\tau$ is defined over $\overline{\mathbb{Q}}$. Let $m$ be a positive integer prime to $N$.
Then, the following statements are true.
\begin{enumerate}
\item For each $m$,
\[
f(T_m.Q_\tau)-\left(\xi_{0}(f|T_m),f_{\psi_D^{alg}}  \right)_{reg}
\]
is algebraic. Especially, if $f$ is a Hecke eigenform, then
\[
f(T_m.Q_\tau)-m^{-1}\lambda_m\left( \xi_{0}(f), f_{\psi_D^{alg}}\right)_{reg}
\]
is algebraic. Here, $\lambda_m$ is the eigenvalue of $\xi_0(f)$ for $T_m$.
\item For each $m$, $f(T_m.Q_\tau)$ is algebraic if and only if $( \xi_0(f|T_m), f_{F_D})$ is algebraic.
\end{enumerate}
\end{lem}

\begin{proof}
(1) By the definition of $\psi_D^{alg}$, the constant term of $f_{\psi_D^{alg}}$ at each cusp is  an integer. Let
\[
E_N(z) :=E_2(z)-NE_2(Nz),
\]
where $E_2(z)=1-24\sum_{n=1}^{\infty} \sigma_1(d) q^n$ is the Eisenstein series of weight $2$. Here, $\sigma_1(n)$ is a function over $\mathbb{Q}$ defined by
$$
 \sigma_1(n):=
\begin{cases}
 \sum_{d|n} d & \text{if $n$ is a positive integer}, \\
  0 & \text{elsewhere}.
\end{cases}
$$
Since $N$ is a prime, $\Gamma_0(N)$ has only two inequivalent cusps. Thus, there is a rational number $c_0$ such that the constant term of
\[
f_{\psi_D^{alg}}- c_0 E_N
\]
is zero at each cusp inequivalent to $i \infty$.
Let us note that $E_N$ is orthogonal to $S_2(\Gamma_0(N))$ with respect to the Petersson inner product, and that
 the principal part of $f_{\psi_D^{alg}}$ at the cusp $i\infty$ is constant since poles of $\psi_D^{alg}$ are simple poles.
By Lemma \ref{regularized_formula}, we have
\begin{align*}
\left( \xi_0(f|T_m), f_{\psi_D^{alg}}\right)_{reg}&=\left(\xi_0(f|T_m), f_{\psi_D^{alg}}- c_0 E_N\right)_{reg}\\
&=\sum_{n\geq 1}a_{f|T_m}(-n)\left(c_{\psi_D^{alg}}(n)- 24 c_0(\sigma(n)-N\sigma(n/N))\right)+(f|T_m)(Q_\tau)\\
&=\sum_{n\geq 1}a_{f|T_m}(-n)\left(c_{\psi_D^{alg}}(n)-24 c_0(\sigma(n)-N\sigma(n/N))\right)+f(T_m.Q_\tau),
\end{align*}
where $a_{f|T_m}(n)$ (resp. $c_{\psi_D^{alg}}$) denotes the $n$-th Fourier coefficient of $(f|T_m)^+$ (resp. $f_{\psi_D^{alg}}$) at the cusp $i\infty$.
Note that the principal part of $f|T_m$ at the cusp $i\infty$ belongs to $\bar{\QQ}[q^{-1}]$
 since the principal part of $f$ at the cusp $i\infty$ does.
This proves the first part.
The second part follows from (\ref{Heckedifferentialcommutative}).

(2) Let us note that $f_{F_D}=f_{\Phi_D}-f_{\psi_D^{alg}}$ is a cusp form.
Thus, by Proposition \ref{regpetersson} we have
$$
 (\xi_0(f|T_m),f_{F_D})=\left(\xi_0(f|T_m), f_{\Phi_D}\right)_{reg}-\left(\xi_0(f|T_m),f_{\psi_D^{alg}}\right)_{reg}=-\left(\xi_0(f|T_m),f_{\psi_D^{alg}}\right)_{reg}.
$$
From (1), we get the desired result.
\end{proof}

Theorem \ref{main0} comes from Lemma \ref{algebraic}.

\begin{proof} [Proof of Theorem \ref{main0}]
It is immediately implied by Lemma \ref{algebraic}.
\end{proof}

For a harmonic weak Maass form $f$, we define
\[
\mathrm{Prin}(f) := \prod_{n>1\atop a_f(-n)\neq 0} n,
\]
where $a_f(n)$ is the $n$-th Fourier coefficient of $f^+$ at the cusp $i\infty$. For a positive integer $m$, we define the operators $U_m$ and $V_m$ for $F$, which are meromorphic functions on $\mathbb{H}$, by
$$(F|U_m)(z)  := \frac{1}{m} \sum_{j=0}^{m-1} F\left( \frac{z+j}{m} \right)$$
and
$$(F|V_m)(z) := F(mz).$$
If $F=\sum_{n}a_F(n)q^n$ is a meromorphic modular form of weight $k$ on $\Gamma_0(N)$, then
\begin{enumerate}
\item $F|U_m$ and $F|V_m$ are meromorphic modular forms of weight $k$ on $\Gamma_0(mN)$,
\item $(F|U_m)(z)=\sum_{n} a_F(mn)q^n$ and $(F|V_m)(z)=\sum_{n} a_F(n)q^{mn}$
\end{enumerate}
(for example, see Section 2.4 in \cite{O} for more details on the operators $U_m$ and $V_m$).
To prove Theorem \ref{main1}, we need the following lemmas.

\begin{lem} \label{nonCM}
Let $f \in H^*_0(\Gamma_0(N))$ be an arithmetic Hecke eigenform.
Let $\tau$ be an algebraic number on $\mathbb{H}$. Assume that $D:=Q_{i \infty}-Q_{\tau}$.
If $\tau$ is not a CM point, then there are infinitely many positive integers $m$ prime to $N\cdot \mathrm{Prin}(f)$ such that
\begin{equation} \label{m_coefficient}
\sum_{n\geq1} a_f(-n)c_{\Phi_D}(mn)
\end{equation}
is transcendental, where $a_f(n)$ (resp. $c_{\Phi_D}(n)$) is the $n$-th Fourier coefficient of $f^+$ (resp. $f_{\Phi_D}$) at the cusp $i\infty$.
\end{lem}

\begin{proof}
Let $A$ be the set of positive integers $n$ such that $a_f(-n)\neq 0$.
Note that $A$ is a finite set.
Then, we define
\[
F(z) := \sum_{n\in A} a_f(-n)(f_{\Phi_D}|U_n)(z).
\]
The $m$-th Fourier coefficient $a_F(m)$ of $F$ at the cusp $i\infty$ is equal to (\ref{m_coefficient}).
Thus, we need to show that $F$ has infinitely many transcendental Fourier coefficients $a_F(m)$ with $(m,N\cdot \mathrm{Prin}(f))=1$.

Let $B = \{p_1,\ldots, p_k\}$ be the set of primes $p$ such that $p | N \cdot \mathrm{Prin}(f)$. By using the operators $U_n$ and $V_n$, we will remove Fourier coefficients $a_F(m)$ with $(m,N\cdot \mathrm{Prin}(f))\neq1$.
We define $F_0 := F$ and $F_{i+1} := F_i - F_i|U_{p_{i+1}}|V_{p_{i+1}}$ for $0\leq i \leq k-1$.
Then, it is enough to show that $F_k$ has infinitely many transcendental Fourier coefficients at the cusp $i\infty$.

By the definition of $f_{\Phi_D}$, we see that $f_{\Phi_D}$ has a singularity at the non-CM point $\tau$ in $\HH$.
We will prove that $F_k$ also has a singularity at a non-CM point in $\HH$.
We define
\[
f_{n,i, \mu_1, \ldots, \mu_k}(z) := f_{\Phi_D}\left( \frac{z}{n} + \frac{i}{n} + \frac{\mu_1}{p_1n} + \cdots +\frac{\mu_k}{p_kn}\right)
\]
for $n\in A, 0\leq i\leq n-1, 0\leq \mu_1\leq p_1-1, \ldots, 0\leq \mu_k\leq p_k-1$.
Note that
\begin{equation*}
\begin{aligned}
F_k(z) &= F(z) - (F|U_{p_1}|V_{p_1})(z) - \cdots -(F|U_{p_k}|V_{p_k})(z) + (F|U_{p_1}|V_{p_1}|U_{p_2}|V_{p_2})(z) \\
&\qquad +\cdots + (F| U_{p_{k-1}}|V_{p_{k-1}}|U_{p_k}|V_{p_k})(z) + \cdots + (-1)^k (F|U_{p_1}|V_{p_1}|\cdots|U_{p_k}|V_{p_k})(z)\\
&= \sum_{n\in A} \sum_{0\leq i\leq n-1} \sum_{0\leq \mu_j \leq p_j-1 \atop 1\leq j\leq k} \alpha_{n,i,\mu_1,\ldots, \mu_k} f_{n,i, \mu_1, \ldots, \mu_k}(z)
\end{aligned}
\end{equation*}
for some nonzero constants $\alpha_{n,i,\mu_1,\ldots, \mu_k}$.
We fix $n_0\in A$.
Let
\[
\beta := n_0\tau - \frac{1}{p_1} - \cdots - \frac{1}{p_k}.
\]
Then, $\beta$ is not a CM point since $\tau$ is not a CM point, and $\beta$ is a singular point of the function $f_{n_0,0,1,\ldots, 1}$.

To prove that $F_k$ has a singularity at $\beta$,
it is enough to show that $\beta$ is not a singular point of $f_{n,i, \mu_1, \ldots, \mu_k}$ if $(n,i,\mu_1,\ldots, \mu_k) \neq (n_0, 0, 1, \ldots, 1)$.
Note that the set of singular points of $f_{n,i, \mu_1, \ldots, \mu_k}$ is a subset of
\[
T_{n,i,\mu_1, \ldots, \mu_k}:= \left \{n(\gamma \tau) - i - \frac{\mu_1}{p_1} - \cdots -\frac{\mu_k}{p_k}\  \bigg|\ \gamma\in \Gamma_0(N)\right\}.
\]
Suppose that $\beta\in T_{n,i,\mu_1, \ldots, \mu_k}$.
Then, for some $\gamma = \sm a&b\\c&d\esm \in \Gamma_0(N)$, we have
\[
n_0\tau - \frac{1}{p_1} - \cdots - \frac{1}{p_k} = n(\gamma \tau) - i - \frac{\mu_1}{p_1} - \cdots -\frac{\mu_k}{p_k}.
\]
This implies that
\[
n_0 \tau - n\frac{a\tau+b}{c\tau+d} = -i - \frac{\mu_1-1}{p_1} - \cdots - \frac{\mu_k-1}{p_k}.
\]
If $c\neq 0$, then $\tau$ satisfies a quadratic equation; this is not possible since $\tau$ is not a CM point.
Thus, $c=0$. Then, we may assume that $a= d= 1$.
From this, we have
\[
n_0\tau - n(\tau+b) = -i - \frac{\mu_1-1}{p_1} - \cdots - \frac{\mu_k-1}{p_k}.
\]
Since $\tau$ is not a rational number, we see that $n = n_0$.
Then, we obtain
\[
-nb + i = - \frac{\mu_1-1}{p_1} - \cdots - \frac{\mu_k-1}{p_k}.
\]
This holds only if $\mu_1= \cdots = \mu_k = 1$ and $b = i  = 0$.
Thus, $\beta$ is a singular point only for $f_{n_0,0,1,\ldots, 1}$.

Let $p$ be a prime.
In a similar argument, we see that $F_k|U_p$ also has a singularity at a non-CM point for every prime $p$.
By the Siegel-Schneider theorem, $F_k|U_p$ is not defined over $\overline{\QQ}$.
Then, the $q$-expansion principle implies there is a Fourier coefficient of $F_k|U_p$ at the cusp $i\infty$ which is transcendental.
Therefore, $F_k$ has infinitely many transcendental Fourier coefficients.
\end{proof}

\begin{lem} \label{notalgebraic}
Let $f\in H^*_0(\Gamma_0(N))$ be an arithmetic Hecke eigenform.
Assume  that $(g, \xi_{0}(f)) \neq 0$ for every Hecke eigenform $g \in S_2(\Gamma_0(N))$.
Let $D$ be a divisor of $X_0(N)$ defined over $\overline{\QQ}$.
Assume that $N$ is a prime, and $\Phi_D$ is not defined over $\overline{\mathbb{Q}}$.
Then, there exist infinitely many positive integers  $m$ prime to $N$ such that $f(T_m.D)$ are transcendental.
\end{lem}

\begin{proof}
Let $\{ f_1, \ldots, f_k \}$ be the set of all normalized Hecke eigenforms in $S_2(\Gamma_0(N))$.
Assume that $f_i$ has a Fourier expansion of the form
\[
f_i(z) = \sum_{n=1}^\infty a_{f_i}(n) e^{2\pi inz}.
\]
Then, we have
\[
( \xi_0(f), f_i)=\sum_{m+n=0} a_{f_i}(m) a_f(n),
\]
where $a_{f}(n)$ is the $n$-th Fourier coefficient of $f^+$ at the cusp $i\infty$.
Thus, $(\xi_0(f), f_i)$ is algebraic

Note that $f_{F_D}$ is a cusp form in $S_2(\Gamma_0(N))$.
Assume that $f_{F_D}=\sum_{i=1}^k \beta_i f_i$ for some $\beta_i$, and that $\xi_0(f) = \sum_{i=1}^k \alpha_i f_i$ for some $\alpha_i$.
By the assumption, $\alpha_i \neq 0$ for all $i$.
Let $m$ be a positive integer prime to $N$.
Then, we have
\[
\xi_0(f|T_m) = m^{-1} \xi_0(f)|T_m = m^{-1} \sum_{i=1}^k \alpha_i \lambda_{i,m}f_i,
\]
where $\lambda_{i,m}$ is the eigenvalue of $f_i$ for $T_m$.
From this, we obtain
\begin{equation*}
\begin{aligned}
(\xi_0(f|T_m), f_{F_D}) &= m^{-1} \sum_{i=1}^k \sum_{j=1}^k (\alpha_i\lambda_{i,m}f_i, \beta_j f_j) = m^{-1}\sum_{i=1}^k \alpha_i \overline{\beta_i} \lambda_{i,m} (f_i, f_i).
\end{aligned}
\end{equation*}
We define
\[
\tilde{\beta}_i := \alpha_i \overline{\beta_i} (f_i, f_i)
\]
for $i=1, \ldots, k$.

Since $\Phi_D$ is not defined over $\overline{\QQ}$, we see that $F_D$ is not defined over $\overline{\mathbb{Q}}$.
This implies that at least one of $\beta_i$ is transcendental.
Note that $\alpha_i (f_i, f_i)$ is a non-zero  algebraic number for all $i$ since  $\alpha_i (f_i, f_i) = (\xi_0(f), f_i)$ and $\alpha_i \neq0$ for all $i$. Thus, at least one of $\tilde{\beta}_i$ is transcendental. As a vector space over $\overline{\mathbb{Q}}$, let $W$ be the subspace of $\mathbb{C}$ generated by $\tilde{\beta}_1,\ldots, \tilde{\beta}_k$.
Let $\{w_1, \ldots, w_\ell \}$ be a basis of $W$, where $w_1=1$.
We may assume that
\begin{equation} \label{beta_equation}
\tilde{\beta}_i=\sum_{j=1}^\ell \beta_{ij}w_j
\end{equation}
for $i=1,\ldots, k$ and for some $\beta_{ij} \in \overline{\mathbb{Q}}$.
From this, we have
\begin{equation} \label{eigenvalue_equation}
(\xi_0(f|T_m), f_{F_D})=  m^{-1}\sum_{i=1}^k \tilde{\beta}_i \lambda_{i,m} = m^{-1}\sum_{i=1}^k \left(\sum_{j=1}^\ell \beta_{i,j}w_j \right) \lambda_{i,m} = m^{-1} \sum_{j=1}^\ell \left( \sum_{i=1}^k \beta_{i,j} \lambda_{i,m}\right) w_j.
\end{equation}
By Lemma \ref{algebraic} (2) and (\ref{eigenvalue_equation}), $f(T_m.D)$ is algebraic if and only if
\begin{equation} \label{linear_equation}
\sum_{i=1}^k   \beta_{i,j}\lambda_{i,m}=0
\end{equation}
for every $j \geq 2$.

Suppose that $f(T_m.D)$ is algebraic for all positive integers $m$ prime to $N$.
Since at least one of $\tilde{\beta}_i$ is transcendental, at least one of $\beta_{i,j}$ with $j\geq2$ is non-zero  by (\ref{beta_equation}).
From this, there is a positive integer $j_0\geq 2$ such that $(\beta_{1, j_0}, \ldots, \beta_{\ell, j_0}) \neq (0,\ldots, 0)$.
We define a cusp form $g$ in $S_2(\Gamma_0(N))$ by
\begin{equation} \label{dfn_g}
g := \sum_{i=1}^k \beta_{i,j_0}f_i.
\end{equation}
Note that by (\ref{linear_equation}), $a_g(m) = 0$ for all positive integers $m$ prime to $N$, where $a_g(m)$ denotes the $m$-th Fourier coefficient of $g$.
This implies that
\[
g|U_N|V_N = g \in S_{2}(\Gamma_0(N))
\]
since $N$ is a prime.
By Lemma 16 in \cite{AL}, $g|U_N$ is a cusp form in $S_2(\Gamma_0(1))$.
Since $S_2(\Gamma_0(1)) = \{0\}$, this implies that $g| U_N = 0$.
From this, we have $g = 0$ since $a_g(m) \neq 0$ only when $N \mid m$.
This is a contradiction due to the fact that $\{f_1, \ldots, f_k\}$ is a basis of $S_2(\Gamma_0(N))$ and  $(\beta_{1, j_0}, \ldots, \beta_{\ell, j_0}) \neq (0,\ldots, 0)$. Thus, there is a positive integer $m_0$ prime to $N$ such that $f(T_{m_0}.D)$ is transcendental.

By (\ref{linear_equation}),
there exists $j_0\geq 2$ such that $\sum_{i=1}^k \beta_{i,j_0}\lambda_{i,m_0}\neq 0$.
If we define $g$ by (\ref{dfn_g}), then
$a_g(m_0) \neq 0$. Thus, the function
\[
(g - g|U_N|V_N)(z) = \sum_{N\nmid m} a_g(m)e^{2\pi imz}
\]
is a non-zero cusp form of weight $2$.
Therefore, there exist infinitely many positive integer $m$ prime to $N$ such that $a_g(m)\neq0$.
With  (\ref{linear_equation}), this completes the proof.
\end{proof}

Now we prove Theorem \ref{main1}.

\begin{proof}[Proof of Theorem \ref{main1}]
First, we prove that if $\tau$ is not a CM point or $nD$ is not rational on $X_0(N)$ for any positive integer $n$, then $f(T_m.Q_\tau)$ is not algebraic for some positive integer $m$ prime to $N$.
Suppose that $\tau$ is not a CM point.
Then, by Lemma \ref{nonCM}, there exists a positive integer $m$ prime to $N\cdot \mathrm{Prin}(f)$ such that
(\ref{m_coefficient}) is transcendental.
Since we have
\[
f(T_m . Q_\tau) = - \sum_{n\geq1} a_f(-n)c_{\Phi_D}(mn)
\]
for positive integers $m$ prime to $N \cdot \mathrm{Prin}(f)$
by Lemma \ref{regularized_formula} and Proposition \ref{regpetersson}, we see that $f(T_m . Q_{\tau})$ is transcendental for some positive integer $m$ prime to $N$.

Suppose that $\tau$ is a CM point and that $nD$ is not rational on $X_0(N)$ for any positive integer $n$.
This implies that $D$ is defined over $\overline{\QQ}$ and that
$D$ is not a principal divisor.
By Theorem \ref{algebraic_differential}, $\Phi_D$ is not defined over $\overline{\QQ}$.
Therefore, there exists a positive integer $m$ prime to $N$ such that $f(T_m. Q_\tau)$ is not algebraic by Lemma \ref{notalgebraic}.

Conversely, suppose that $\tau$ is a CM point and $nD$ is rational on $X_0(N)$ for some positive integer $n$.
Then, $D$ is defined over $\overline{\QQ}$ and $D$ is a principal divisor.
By Theorem \ref{algebraic_differential}, $\Phi_D$ is defined over $\overline{\QQ}$. Thus, all the Fourier coefficients of $f_{F_D}$ at the cusp $i\infty$  is algebraic.
Let $m$ be a positive integer prime to $N$.
Then, by Lemma \ref{regularized_formula}, we have
\[
(\xi_{0}(f|T_m), f_{F_D}) = \sum_{n\geq1} a_{f|T_m}(-n) c_{F_D}(n),
\]
where $a_{f|T_m}(n)$ (resp. $c_{F_D}(n))$ is the $n$-th Fourier coefficient of $(f|T_m)^+$ (resp. $f_{F_D}$) at the cusp $i\infty$.
This implies that $(\xi_{0}(f|T_m), f_{F_D})$ is algebraic.
Therefore, by Lemma \ref{algebraic} (2), $f(T_m.Q_\tau)$ is algebraic.
\end{proof}


\end{document}